\documentclass[a4paper,11pt,reqno]{amsart}
\usepackage{amssymb}
\usepackage{enumitem}
\usepackage[alphabetic]{amsrefs}
\usepackage{tikz}
\usepackage{blindtext, xcolor}
\usepackage{comment}
\usepackage{mathrsfs}
\usepackage{geometry}
\usepackage{a4wide}
\usepackage{fullpage}
\usepackage{mathtools}

\makeatletter
\@namedef{subjclassname@2020}{%
  \textup{2020} Mathematics Subject Classification: 5C31 primary; 82B20, 68W25 secondary} 
\makeatother

\newtheorem{corollary}{Corollary}[section]
\newtheorem{theorem}[corollary]{Theorem}
\newtheorem{lemma}[corollary]{Lemma}
\newtheorem{proposition}[corollary]{Proposition}

\newtheorem{claim}[corollary]{Claim}
\newtheorem{remark}[corollary]{Remark}

\newtheorem*{theorem*}{Theorem}
\newtheorem*{corollary*}{Corollary}

\numberwithin{equation}{section}

\usepackage{graphicx} 

\title{Improved bounds for the zeros of the chromatic polynomial via Whitney's Broken Circuit Theorem}
\author[M. Jenssen]{Matthew Jenssen}
\address{Matthew Jenssen, Department of Mathematics, King's College London, Strand, London,
WC2R 2LS}
\email{\texttt{matthew.jenssen@kcl.ac.uk}}
\author[V. Patel]{Viresh Patel}
\address{Viresh Patel, School of Mathematical Sciences, Queen Mary University of London, Mile End Road, London, E1 4NS}
\email{\texttt{viresh.patel@qmul.ac.uk}}
\author[G.Regts]{Guus Regts}
\address{Guus Regts, Korteweg de Vries Institute for Mathematics, University of Amsterdam. P.O. Box 94248  
1090 GE Amsterdam The Netherlands}
\email{\texttt{guusregts@gmail.com}}
\thanks{MJ is supported by a UK Research and Innovation Future Leaders Fellowship MR/W007320/2.}
\thanks{GR was funded by the Netherlands Organisation of Scientific Research (NWO): VI.Vidi.193.068}
\date{\today}

\begin{document}

\maketitle
\begin{abstract}
We prove that for any graph $G$ of maximum degree at most $\Delta$, the zeros of its chromatic polynomial $\chi_G(x)$ (in $\mathbb{C}$) lie inside the disc of radius $5.94 \Delta$ centered at $0$. This improves on the previously best known bound of approximately $6.91\Delta$.

We also obtain improved bounds for graphs of high girth. We prove that for every $g$ there is a constant $K_g$ such that for any graph $G$ of maximum degree at most $\Delta$ and girth at least $g$, the zeros of its chromatic polynomial $\chi_G(x)$ lie inside the disc of radius $K_g \Delta$ centered at $0$, where $K_g$ is the solution to a certain optimization problem. In particular, $K_g < 5$ when $g \geq 5$ and $K_g < 4$ when $g \geq  25$ and $K_g$ tends to approximately $3.86$ as $g \to \infty$. 

Key to the proof is a classical theorem of Whitney which allows us to relate the chromatic polynomial of a graph $G$ to the generating function of so-called broken-circuit-free forests in $G$. We also establish a zero-free disc for the generating function of all forests in $G$ (aka the partition function of the arboreal gas) which may be of independent interest.  
\end{abstract}

\section{Introduction}
The chromatic polynomial $\chi_G(x)$ of a graph $G=(V,E)$ can be defined by
\[\chi_G(x)=\sum_{A\subseteq E} (-1)^{|A|} x^{\kappa(A)},\]
where $\kappa(A)$ denotes the number of components of the graph $(V,A)$.
This polynomial has the property that, when evaluated at a positive integer $q$, it counts the number of proper $q$-colourings of the graph $G$.

Birkhoff~\cite{Birk} introduced the chromatic polynomial (for planar graphs) and was interested in its zeros, henceforth called \emph{chromatic zeros}, as a way of approaching (what was then) the four colour conjecture. So far this approach has not been successful; however since then numerous results have appeared concerning real chromatic zeros of planar graphs~\cites{BirkhoffLewis46,Tuttegoldenratio,Jackson93,Thomassen97,Royle08,ThomassenPerret2018}, with still several open questions remaining~\cites{ThomassenPerret2018,handbook}.
Sokal~\cite{Sokaldense} showed that the complex chromatic zeros lie dense in the complex plane, while this remains open for chromatic zeros of planar graphs.
Chromatic zeros have been extensively studied for other families of graphs as well; we refer the reader to~\cites{sokal,Jackson93,handbook} and the references therein for a comprehensive overview. 

In the present paper we are interested in the chromatic zeros for bounded degree graphs.
Brenti, Royle and Wagner~\cite{BRW94} asked if there exists some function $f: \mathbb{N} \rightarrow \mathbb{R}$ such that for every graph of maximum degree $k$, its (complex) chromatic zeros lie inside the disc of radius $f(k)$ centered at $0$. This extended an earlier conjecture of Biggs, Darmerell and Sands~\cite{BDS72} stated only for regular graphs.
Sokal~\cite{sokal} answered these questions affirmatively in a strong sense by showing that there exists an absolute constant $K \approx 7.97$  with the property that, for every graph $G$, the complex roots of $\chi_G$ lie inside the disc of radius $K \Delta(G)$ centered at $0$, where $\Delta(G)$ is the maximum degree of $G$. 
Borgs~\cite{Borgs} later provided a short alternative proof of Sokal's result. 
The constant was subsequently improved to $K \approx 6.91$ by Fern{\'a}ndez and Procacci~\cite{Fernandez}, and Jackson, Procacci, and Sokal~\cite{JPSzeros}. 
Our main contribution is an improvement on these bounds.
\begin{theorem} \label{thm:main}
 There exists a constant $K\leq 5.94$ such that for every graph $G$, $\chi_{G}(x)\ne 0$ for all $z\in \mathbb{C}$ with $|x|\ge K\Delta(G)$.
\end{theorem}
For graphs of large girth we can obtain better bounds.
\begin{theorem} \label{thm:large girth}
For each $g\geq 3$ there exists a constant $K_g$ such that $\lim_{g\to \infty}K_g\leq 3.86$ and for every graph $G$ of girth at least $g$, $\chi_{G}(x)\ne 0$ for all $x\in \mathbb{C}$ with $|x|\ge K_g\Delta(G)$. 
\end{theorem}
We will bound $K_g$ by estimating the optimal solution to a certain constrained optimization problem. We refer to Figure 1 below for concrete bounds on $K_g$ for several values of $g$.


Note that, since there are graphs of maximum degree $\Delta$ with no proper $\Delta$-colouring, we cannot take $K<1$ in Theorem~\ref{thm:main}. In fact, in general, we cannot take $K < 1.599$ (see Section 9 of \cite{Sokalsurvey}). 
However, it is conjectured by Sokal (see \cite[Conjecture 21]{Jackson}, and Sokal~\cite[Conjecture 9.5'']{Sokalsurvey}) that $\chi_G(x) \not= 0$ if ${\rm Re}(x) > \Delta(G)$. A positive resolution of this conjecture would resolve a major open problem in the field of approximate counting: it would imply (see \cite{PR17}) the existence of a fully polynomial time approximation scheme for counting proper $q$-colourings of a graph of maximum degree $\Delta$ for $q\geq \Delta+1$. 
We note that a mildly stronger conjecture of Sokal, \cite[Conjecture 9.5']{Sokalsurvey} allowing one vertex of unbounded degree in the previous conjecture, is not true.\footnote{Counterexamples for $
\Delta=3$ were found by Royle, see~\cite[Footnote 31]{Sokalsurvey}. More recently, counterexamples were found for all but a finite range of values of $\Delta$~\cite{BHRseriesparallelroots23}.}

Our approach, which we detail in the next section, motivates the study of the roots of the \emph{forest generating function} of a graph $G=(V,E)$ defined as
\begin{align}\label{eq:forestgnfn}
    Z_G(x) = \sum_{\substack{F\subseteq E: \\(V,F) \text{ is a forest}}} x^{|F|}\, .
\end{align}

$Z_G(x)$ is the partition function of the \emph{arboreal gas} which has been studied significantly in the statistical physics and probability literature in recent years~\cites{halberstam2023uniqueness, bauerschmidt2021percolation, bauerschmidt2021random, caracciolo2004fermionic, caracciolo2017spanning, deng2007ferromagnetic}. We prove the following result for the roots of the forest generating function on bounded degree graphs. 

\begin{theorem}\label{thm:forest}
Let $G=(V,E)$ be a graph and let $v\in V$. Suppose that all vertices of $G$, except possibly $v$ have degree at most $\Delta\in \mathbb{N}$.
Then for any $x$ such that $|x|\leq \tfrac{1}{2\Delta}$, $Z_G(x)\neq 0$.
\end{theorem}

Zero-freeness results such as the one above are of interest in statistical physics as they imply the absence of phase transitions in the Lee-Yang sense~\cite{yang1952statistical}. It would be interesting to determine whether the $1/2$ in  Theorem~\ref{thm:forest} is optimal and whether there is a zero-free region containing $0$ that contains a larger portion of the positive real axis. We note that Theorem 1.3 from~\cite{BencCsikvari22} implies that $1/(2 \Delta)$ in the above theorem cannot be replaced by $\tfrac{1}{\Delta-2}$, indeed it implies there cannot be an open zero-free region in the complex plane containing the real interval $[0,\tfrac{1}{\Delta-2}]$. 

We remark that all graphs are assumed to be simple unless stated otherwise.

\subsection{Approach}
Our approach is inspired by the approach taken in~\cites{sokal,Borgs,Fernandez,JPSzeros} but differs crucially in our use of Whitney's Theorem for describing the chromatic polynomial. 
Using the language of statistical physics, the approach in~\cites{sokal,Borgs,Fernandez,JPSzeros} is to relate the chromatic polynomial to the \emph{partition function} of a \emph{polymer model}. A polymer model consists of three ingredients: (i) a (finite) set $\mathcal{P}$ whose elements we call polymers; (ii) a symmetric and reflexive `compatibility' relation $\sim$ on $\mathcal{P}$; and (iii) a weight function $\omega: \mathcal{P}\to\mathbb{C}$. The partition function of this polymer model is then
\[
\Xi(\mathcal{P})=\sum_{\Gamma} \prod_{\gamma\in\Gamma} w(\gamma)\, ,
\]
where the sum is over all collections $\Gamma$ of mutually compatible polymers. 
The motivation for relating the chromatic polynomial to such a partition function is that one can then apply the extensive theory of polymer models from statistical physics; in particular results that constrain the location of zeros of $\Xi(\mathcal{P})$ considered as a function of the complex weights $w(\gamma)$. For example, Sokal~\cite{sokal} and Borgs~\cite{Borgs} apply Dobrushin's classical zero-freeness criterion~\cite{dobrushin1996estimates} whereas Fern{\'a}ndez and Procacci applied their own more recent zero-freeness criterion~\cite{fernandez2007cluster} and Jackson, Sokal and Procacci~\cite{JPSzeros} used the Gruber-Kunz criterion~\cite{BFPgruberkunzbound}. The main difficulty in these proofs is to verify that the conditions of the zero-freeness criteria hold. 

In each of the aforementioned papers the set of polymers $\mathcal{P}$ is taken to be vertex subsets of $G$ that induce a connected subgraph of $G$. We take a different approach by applying Whitney's classical Broken Circuit Theorem (see below) to relate the chromatic polynomial of $G$ to the partition function of a model in which polymers are the edge sets of so-called `broken-circuit-free' trees in $G$. 
We moreover do not directly rely on any results from statistical physics that constrain the location of zeros of $\Xi(\mathcal{P})$, but rather use the proofs of these results as inspiration. 
Indeed, inspired by the proof of~\cite[Proposition 3.1]{BFPgruberkunzbound}, we give a direct and new inductive proof for the zero-freeness of the chromatic polynomial that takes advantage of the tree structure of polymers as well as the specific structure of broken-circuit-free trees. 
To go into more detail we now introduce Whitney's theorem. 

\subsubsection{Whitney's Broken Circuit Theorem} A classical theorem of Whitney~\cite{Whitney} gives a combinatorial description of the coefficients of the chromatic polynomial in terms of so-called \emph{broken-circuit-free sets}. For a graph $G = (V,E)$ fix any ordering of $E$. A \emph{broken circuit} in the graph $G$ is obtained by taking the edges of any cycle of $G$ and removing its largest edge (in the given ordering). A \emph{broken-circuit-free set} in $G$ is any set of edges $F \subseteq E$  that does not contain a broken circuit. Equivalently, $F \subseteq E$ is broken-circuit free if and only if
\begin{itemize}
\item $F$ has no cycles (i.e.\ is a forest) and 
\item for every $e \in E \setminus F$, if $F \cup e$ contains a cycle (which is then necessarily unique and contains $e$), then $e$ is not the largest edge of that cycle. 
\end{itemize}
From now on, we abbreviate ``broken circuit free'' by BCF.
Let us write $\mathcal{F}_G$ for the set of all BCF sets in $G$ (including the empty set). 
We define the polynomial
\begin{align}\label{eq:FGDef}
    F_G(x) = \sum_{F \in \mathcal{F}_G}x^{|F|}.
\end{align}
While $\mathcal{F}_G$ depends on the ordering of the edges, the polynomial $F_G$ remarkably does not.
Moreover, the polynomial $F_G$ is up to a simple transformation equal to the chromatic polynomial: 

\begin{theorem}[Whitney \cite{Whitney}]\label{thm:whitney}
If $G$ is an $n$-vertex graph, then 
\[
\chi_G(x) = x^nF_G(-1/x). 
\]
\end{theorem}
We give a short proof of this theorem at the end of this section to make the paper self contained.

To make the connection to the polymer model approach explicit, we note that $F_G(x)$ is the partition function of a polymer model in which polymers are BCF trees in $G$; two polymers are compatible if their union in $G$ is disconnected; and the weight $w(T)$ of a polymer $T$ is $x^{|T|}$ where $|T|$ denotes the number of edges in $T$.

We note that by Whitney's theorem, showing that $\chi_G(x) \not = 0$ for $|x| \geq R$ is equivalent to showing $F_G(x) \not = 0$ for $|x| \leq 1/R$. Proving such a zero-free disc for $F_G$ will be the main goal of the paper. 
We note that our proof also implicitly gives a zero-free disc around $0$ for the forest generating function of $G$, as is also the case in~\cites{sokal,Borgs,Fernandez,JPSzeros}.
However, with a different proof we can obtain a significantly better bound which is the content of Theorem~\ref{thm:forest}. This is detailed in Section~\ref{sec:forests}. 

We conclude this section with a proof of Whitney's theorem. In Section~\ref{sec:prelim} we collect some combinatorial preliminaries. In Section~\ref{sec:main result} we prove our main results, Theorems~\ref{thm:main} and~\ref{thm:large girth}. In Section~\ref{sec:forests} we prove Theorem~\ref{thm:forest}. We end with some concluding remarks and directions of future research. 

\begin{proof}[Proof of Theorem~\ref{thm:whitney}]
Let $G$ be an $n$-vertex graph and choose any ordering of the edges of $G$.
Our goal is to prove that 
\begin{equation}\label{eq:induction claim forest}
\chi_G(x)=x^nF_G(-1/x).    
\end{equation}
We now use induction on the number of edges to prove~\eqref{eq:induction claim forest}.
If $G$ has no edges, then $x^nF_G(-1/x)=x^n=\chi_G(x)$, showing the base case.
Next assume that $G$ has at least one edge and fix an ordering of the edges of $G$. Let $e$ be the smallest edge in the ordering. 
Next let $G/e$ denote the simple graph obtained by contracting the edge $e$ where if multiple edges arise we keep only the edge that is largest in the ordering. The ordering of the edges of $G/e$ is taken to be the one induced by the ordering on $G$. Let $\mathcal{F}_{G,e}$ denote the BCF sets of $G$ that contain the edge $e$. Note that since $e$ is the smallest edge of $G$, $\mathcal{F}_{G,e}$ corresponds bijectively with $\mathcal{F}_{G/e}$ via $F\mapsto F\setminus e$. Indeed, if $e$ is contained in a triangle only the largest of the other two edges in that triangle can be contained in a BCF set that contains $e$.

Therefore
\begin{align*}
x^nF_G(-1/x)&=\sum_{F \in \mathcal{F}_{G,e}}(-1)^{|F|}x^{n-|F|}+\sum_{F\in \mathcal{F}_{G\setminus e}}(-1)^{|F|}x^{n-|F|}
\\
&=-\sum_{F \in \mathcal{F}_{G,e}}(-1)^{|F|-1}x^{(n-1)-(|F|-1)}+\sum_{F\in \mathcal{F}_{G\setminus e}}(-1)^{|F|}x^{n-|F|}
\\
&=-\sum_{F \in \mathcal{F}_{G/e}}(-1)^{|F|}x^{(n-1)-|F|}+\sum_{F\in \mathcal{F}_{G\setminus e}}(-1)^{|F|}x^{n-|F|}
\\
&=- x^{n-1}F_{G/e}(-1/x)+x^nF_{G\setminus e}(-1/x)=-\chi_{G/e}(x)+\chi_{G\setminus e}(x)=\chi_G(x),
\end{align*}
where the penultimate equality follows by induction, since both $G/e$ and $G\setminus e$ have fewer edges than $G$ and the last equality is by the well-known deletion contraction formula for the chromatic polynomial.
This completes the proof.
\end{proof}


\section{Preliminaries}\label{sec:prelim}
In this section we collect some combinatorial preliminaries that will be used in the proof of Theorems~\ref{thm:main} and~\ref{thm:large girth}.
\subsection{Tree generating functions}
An acyclic connected graph is called a tree. For a tree $T$, we abuse notation by using $T$ to refer both to the graph and its edge set. 
For a graph $G=(V,E)$, a vertex $v\in V$, and a variable $x$, we define the \emph{rooted tree generating function}  by
\[
T_{G,v}(x):=\sum_{\substack{T\subseteq E(G) \text{ a tree, } \\ v\in V(T)}} x^{|T|}.
\]
The following lemma shows how to bound the rooted tree generating function; it appears somewhat implicitly in~\cite{JPSzeros}. 
The proof we give is taken from~\cite{PRsurvey} and we include it for completeness.
\begin{lemma}
\label{lem:tree}
Let $G=(V,E)$ be a graph of maximum degree at most $\Delta\geq 1$ and let $v\in V$.
Fix any $\alpha\geq 1$,
then 
\[
T_{G,v}\left(\tfrac{\ln \alpha}{\alpha \Delta}\right)\leq \alpha.
\]
\end{lemma}
Note that for $\alpha \geq 1$, $\ln \alpha / (\alpha \Delta)$ takes values in the range $[0, (e\Delta)^{-1}]$  and so the lemma bounds $T_{G,v}(x)$ when $|x|\leq (e\Delta)^{-1}$. The example of the infinite $\Delta$-regular tree shows that it is not possible to give a finite uniform bound on $T_{G,v}(x)$ for $x$ outside this interval in general.

\begin{proof}
The proof is by induction on the number of vertices of $G$. If $|V| = 1$, the statement is clearly true.
Next assume that $|V|\geq 2$.
Given a tree $T$ such that $v\in V(T)$, let $S$ be the set of neighbours of $v$.
After removing $v$ from $T$, the tree decomposes into the disjoint union of $|S|$ trees, each containing a unique vertex from $S$.
Therefore, writing $c=\tfrac{\ln \alpha}{\alpha \Delta}$, we have
\[
T_{G,v}(c)\leq \sum_{S\subseteq N_G(v)} c^{|S|}\prod_{s\in S}T_{G-v,s}(c),
\]
which by induction is bounded by
\[
\sum_{S\subseteq N_G(V)} (c\alpha)^{|S|} \leq (1+(\ln \alpha)/\Delta)^\Delta\leq e^{\ln\alpha}=\alpha.
\]
This completes the proof.
\end{proof}

\begin{remark}
In the above lemma and in the proofs to follow we use bounds such as $(1+1/\Delta)^\Delta\leq e$ that are tight only in the large $\Delta$ limit; moreover at various places $\Delta$ can be replaced by $\Delta-1$ (and sometimes even $\Delta-2$).  
By being more careful, one could obtain an improved zero-free region in Theorem~\ref{thm:main} for small $\Delta$.
We use the simpler bounds for clarity of presentation and because these improvements do not affect the value of $K$ in the general statement of Theorem~\ref{thm:main}, which is our main concern.
\end{remark}

Next we prove a similar result for \emph{double-rooted trees}.
For a graph $G=(V,E)$, two distinct vertices $v_1, v_2\in V$, and a variable $x$, we define the \emph{double-rooted tree generating function} by
\[
T_{G,v_1, v_2}(x):=\sum_{\substack{T\subseteq E(G) \text{ a tree, } \\ v_1, v_2\in V(T)}} x^{|T|}.
\]
If $v_1$ is disconnected from $v_2$ in $G$, then we define $T_{G,v_1, v_2}(x)$ to be the zero polynomial in $x$.

\begin{lemma}
\label{lem:DR-tree}
Let $G=(V,E)$ be a graph of maximum degree at most $\Delta\geq 1$ and let $v_1, v_2\in V$ be distinct vertices.
Fix any $\alpha \in [1, e)$.
Then 
\[
T_{G,v_1, v_2}\left(\tfrac{\ln \alpha}{\alpha \Delta}\right)
\leq D(\alpha) := \tfrac{\alpha \ln \alpha}{\Delta (1 - \ln \alpha)}.
\]
Furthermore, if there is no path with $g$ or fewer vertices between $v_1$ and $v_2$ in $G$, then 
\[
T_{G,v_1, v_2}\left(\tfrac{\ln \alpha}{\alpha \Delta}\right) \leq (\ln \alpha)^{g-1}D(\alpha).
\]
\end{lemma}

\begin{proof}
Note that every tree in $G$ that contains $v_1$ and $v_2$ can be obtained by first taking a path $P$ from $v_1$ to $v_2$ and then appending trees to each vertex on the path (although this also creates objects that are not trees since the trees we attach to $P$ may intersect $P$ elsewhere or intersect other appended trees). Let us write $p(i)$ for the number of paths from $v_1$ to $v_2$ with $i$ vertices and note that $p(i) \leq \Delta^{i-2}$ (in building the path we have at most $\Delta$ choices for each vertex except the first and last). 
Suppose that $T(x)$ is a uniform upper bound for $T_{G,v}(x)$ over $v \in V$ such that $\Delta xT(x)< 1$.  Then, 
\[
T_{G,v_1, v_2}(x) \leq \sum_{i \geq 2} p(i)x^{i-1} T(x)^{i} \leq \sum_{i \geq 2} \Delta^{i-2} x^{i-1} T(x)^{i} = xT(x)^2(1 - \Delta x T(x))^{-1}.
\]
From Lemma~\ref{lem:tree}, we can take $T\left(\tfrac{\ln \alpha}{\alpha \Delta}\right) = \alpha$, and so we see that 
\[
T_{G,v_1, v_2}\left(\tfrac{\ln \alpha}{\alpha \Delta}\right) \leq \frac{\alpha \ln \alpha}{\Delta} (1 - \ln \alpha)^{-1} = D(\alpha).
\]

If there is no path with $g$ or fewer vertices from $v_1$ to $v_2$ in $G$ then 
\[
T_{G,v_1, v_2}(x) \leq \sum_{i \geq g+1} p(i)x^{i-1} T(x)^{i} \leq \sum_{i \geq g+1} \Delta^{i-2} x^{i-1} T(x)^{i} = (\Delta x T(x))^{g-1} \sum_{i \geq 2} \Delta^{i-2} x^{i-1} T(x)^{i} 
\]
so that
\[
T_{G,v_1, v_2}\left(\tfrac{\ln \alpha}{\alpha \Delta}\right) \leq (\ln \alpha)^{g-1}D(\alpha).
\]
\end{proof}

We also need the following counting lemma.
\begin{lemma}\label{lem:derivcount}
Let $X = \{1, \ldots, k\}$ and let $x$ be a variable. Then
\[ 
\sum_{S \subseteq X} \sum_{\substack{(s,t)\in S\times X \\ t>s}} x^{|S|}
= \sum_{S \subseteq X}x^{|S|} \sum_{i \in S}(k-i) = \binom{k}{2}x(1+x)^{k-1}.
\]
\end{lemma}
\begin{proof}
For each $r=0, \ldots, k$, the coefficient of $x^r$ is given by 
\[
\sum_{|S|=r} \sum_{i \in S}(k-i) 
= \sum_{i=1}^k \sum_{\substack{S: \; i \in S \\ |S|=r }}(k-i)
=\sum_{i=1}^k \binom{k-1}{r-1}(k-i)
= \binom{k-1}{r-1}\binom{k}{2} = \frac{1}{2}r(k-1)\binom{k}{r}.
\]
So
\[
\sum_{S \subseteq X}x^{|S|} \sum_{i \in S}(k-i) = \sum_{r=0}^k \frac{1}{2}(k-1)r\binom{k}{r}x^r = \frac{1}{2}(k-1)x \frac{d}{dx}(1+x)^k = \binom{k}{2}x(1+x)^{k-1}.
\]
\end{proof}

\subsection{Broken-circuit-free sets}\label{sec:BCFPrelim}

In this section, we establish some notation and simple properties related to BCF sets that will be used in the proof of our main result.

As before, let $G=(V,E)$ be a graph and recall that $\mathcal{F}_G$ is the set of all BCF sets of $G$. For a set of vertices $S \subseteq V$, we write $\mathcal{F}_{G,S}$ for those $F \in \mathcal{F}_G$ such that every non-trivial component of $(V,F)$ (i.e.\ a component that is not an isolated vertex) contains at least one vertex of $S$. We note that $\mathcal{F}_{G,S}$ contains the empty set of edges. So, for example, if $S$ is just a single vertex $u$, then $\mathcal{F}_{G,u}$ is the set of (edge sets of) BCF trees that contain $u$ along with the empty set of edges. 
We write $\mathcal{F}^*_{G,S}$ for those $F \in \mathcal{F}_{G}$ such that each non-trivial component of $F$ contains exactly one vertex of $S$ (so $\mathcal{F}^*_{G,S} \subseteq \mathcal{F}_{G,S}$).
For two distinct vertices $s,t \in V$, we write $\mathcal{F}_{G,s,t}$ for those $F \in \mathcal{F}_G$ that have exactly one non-trivial component and that component contains $s$ and $t$, i.e. $\mathcal{F}_{G,s,t}$ is the set of (edge sets of) BCF trees that contain $s$ and $t$.
For any element $F \in \mathcal{F}_G$, we write $V(F)$ to mean the set of non-isolated vertices of the graph $(V,F)$. 

The following inequality will be useful. For any positive number $y$,
\begin{equation}
\label{eq:Forest-to-tree}
    \sum_{F \in \mathcal{F}_{G,S}}y^{|F|} \leq \prod_{s \in S} \sum_{T \in \mathcal{F}_{G,s}}y^{|T|} \leq \prod_{s \in S}T_{G,s}(y),
\end{equation}
which holds because $\mathcal{F}_{G,S} \subseteq \{\cup_{s \in S}T_s: T_s \in \mathcal{F}_{G,s}\}$.

For $U \subseteq V$ we will write $\mathcal{F}_{U}$ instead of $\mathcal{F}_{G[U]}$ and extend this notation in the natural way to other situations. As usual $G[U]$ denotes the induced subgraph of $G$ on vertex set $U$.

The following proposition captures the crucial property of BCF sets that we make use of in our main results. 

\begin{proposition}
\label{pr:BFStree}
Let $G=(V,E)$ be a graph, let $u \in V$, $S\subseteq N(u)$ and set $V' = V-u$. Suppose $E$ has an ordering such that the edges incident to $u$ are largest in the order. Given a forest $F \in \mathcal{F}^*_{V', S}$, we have that $T = F \cup \{us: s \in S\}$ is a tree and $T$ contains a broken circuit if and only if there is some $s \in S$ such that the non-trivial component of $F$ containing $s$ also contains some $t \in N(u)$ with $us < ut$ in the ordering of $E$ .
\end{proposition}
\begin{proof}
The definition of $ \mathcal{F}^*_{V', S}$ implies that $T$ is a tree with $V(T) = V(F) \cup \{u\}$. If the non-trivial component of $F$ containing some $s \in S$ also contains $t \in N(u)$ with $us < ut$, then there is a path in $T$ from $u$ to $t$ and this path is a broken circuit. Conversely, if $T$ contains a broken circuit, let $e = ab \in E \setminus T$ be such that $T \cup e$ contains a cycle $C$ in which $e$ is the largest edge. We must have $a, b \in V(T)$. Suppose $a, b \in V(F)$.
We cannot have that $a$ and $b$ belong to the same non-trivial component of $F$ since $F$ is BCF. We also cannot have $a$ and $b$ in different non-trivial components of $F$ since the path from $a$ to $b$ in $T$ then contains $u$ so that $e=ab$ cannot be the highest edge in the cycle $C$. So it must be the case that $u$ is one of the endpoints of $e$, i.e. $e = ut$ with $t \in V(F)$. Let $s$ be the unique vertex of $S$ in the same component as $t$ (note that $s \not= t$ since $e \not\in T$). Then the broken circuit in $T$ must be the path $P$ in $T$ from $u$ to $t$ and $C = P \cup ut$. We know the first edge of $P$ is $us$, so we must have $us < ut$.
\end{proof}

\section{Proof of the main results}\label{sec:main result}
In this section we prove our main results Theorems~\ref{thm:main} and~\ref{thm:large girth}. The results will follow from the following theorem which establishes a zero-free disc for $F_G$ (defined at~\eqref{eq:FGDef}) whose radius can be estimated via a constrained optimization problem.

\begin{theorem}\label{thm:fgbd}
Fix $g \in \mathbb{N}_{\geq 3}$ (which will represent girth).
For constants $a \in [0,1)$ and $b\in [1,e)$, define $h(a,b) := \tfrac{(1-a) \ln b}{b}$ and
\begin{align*}
f_g(a,b):&= \exp\left(\frac{(1-a)\ln b}{b}\right) - 1 + \frac{b\cdot(\ln b)^{g-1}}{2(1 - \ln b)}.
\end{align*}
For any choice of $a$ and $b$ satisfying $f_g(a,b) \leq a$ 
the following holds. 
For any graph $G$ of girth at least $g$ and maximum degree at most $\Delta \geq 1$ 
and any $x \in \mathbb{C}$ satisfying $|x| \leq h(a,b)/\Delta$, we have that
$F_G(x) \not= 0$ and so  $\chi_G(x) \not = 0$ for any $x \in \mathbb{C}$ satisfying $|x| \geq \Delta/h(a,b)$.
\end{theorem}

For some fixed values of $g$, the following table gives values of $a$ and $b$ for which the above inequalities hold. The first line of the table together with Theorem~\ref{thm:fgbd} proves Theorem~\ref{thm:main}.
\begin{figure}[h]
\begin{tabular}{ |p{3cm}||p{3cm}|p{3cm}|p{3cm}|  }
 \hline
 $g$& $a$ & $b$ & $h(a,b)^{-1}$\\
 \hline
3 & 0.329963 &1.434786  & 5.93148\\
4 &0.307493  & 1.521606 & 5.23445\\
5 &0.298048  & 1.593365 & 4.87264\\
6 &0.293619  & 1.654298 & 4.65234\\
7 &0.291435  & 1.707032 & 4.50511\\
8 &0.290359  & 1.753334 & 4.40009\\
9 &0.289864  & 1.794457 & 4.32173\\
10 &0.289689  & 1.831325 & 4.26121\\
15 &0.290356  & 1.971723 & 4.09258\\
20 &0.291349  & 2.06735 & 4.01683\\
25 &0.292139  & 2.13778 & 3.97497\\
100 &0.29505  & 2.4704 & 3.87487\\
$g \to \infty$ & 0.295741...  & 2.71828... & 3.85977...\\
 \hline
\end{tabular}
\caption{Values for $a$ and $b$ and $g$ for which $f_g(a,b)<a$ and corresponding values for $h(a,b)^{-1}$. The rightmost column gives concrete bounds on the value of $K_g$ in the statement of Theorem~\ref{thm:large girth}. For example, $K_5\leq 4.87264$ and $K_{25}\leq 3.97497$.}
\end{figure}\label{tab:bounds}

\begin{proof}[Proof of Theorem~\ref{thm:fgbd}]
First we establish some notation and a few facts we will need. 
For any graph $G=(V,E)$ and a set of vertices $U \subseteq V$, recall that we write $F_U(x)$   to mean $F_{G[U]}(x)$. 

The proof is based on applying induction by using the following recursion. Given any graph $G=(V,E)$ (with some ordering of the edges of $G$ so that the notion of BCF is well defined) and any $S \subseteq V$, we have
\begin{equation}
\label{eq:recurF}
F_V(x) = \sum_{F \in \mathcal{F}_{G,S}} x^{|F|} F_{V- S - V(F)}(x)
\end{equation}
since any $F \in \mathcal{F}_G$ can be uniquely determined by first identifying those components of $F$ that contain vertices in $S$ and then identifying the remainder of the forest and noting that $F$ is BCF if and only if every component of $F$ is BCF. 
Now, specialising~\eqref{eq:recurF} to the case when $S$ is a single vertex $u\in V$, and separating the contribution to the sum from $F=\emptyset$, we see that
\begin{equation}
\label{eq:recur}
F_V(x) = F_{V-u}(x) + \sum_{\substack{T \in \mathcal{F}_{G,u} \\ |T| \geq 1 }}x^{|T|}F_{V-u -V(T)}(x)\, .
\end{equation}
Dividing through by $F_{V-u}(x)$, we define the rational function $R(x)$ by 
\begin{equation}
\label{eq:R}
R = R(x):= \frac{F_V(x)}{F_{V-u}(x)} - 1 = \sum_{\substack{T \in \mathcal{F}_{G,u} \\ |T| \geq 1 }}x^{|T|}\frac{F_{V-u-V(T)}(x)}{F_{V-u}(x)}\, .
\end{equation}

Note that the sums in all three identities above depend on how we order the edges of $G$, but that all the expressions are true for any choice of ordering. 
The following claim immediately implies the theorem, since $F_V(x)=F_G(x)$.

\begin{claim}
Assume $a \in [0,1)$ and $b\in [1,e)$ satisfy $f_g(a,b) \leq a$ 
as in the statement of Theorem~\ref{thm:fgbd}.
    For any graph $G=(V,E)$ of girth at least $g$ and maximum degree at most $\Delta$, 
    any vertex $u \in V$ and any complex number $x$ satisfying $|x| \leq h(a,b) / \Delta$,
    we have 
    \begin{itemize}
        \item[(i)] $F_{V}(x)\neq 0$,
        \item[(ii)] $|R(x)| = |F_V(x)/F_{V-u}(x) - 1| \leq a$.
    \end{itemize}
\end{claim}

We now prove the claim by induction on the number of vertices in $G=(V,E)$. 
If $|V|=2$ then $F_V(x)$ is either $1$ or $(1+x)$ whereas $F_{V-u}(x) = 1$ for any $u \in V$, so that $|R| \leq |x| \leq h(a,b) / \Delta \leq h(a,b) \leq f_g(a,b) \leq a$ 
and hence $F_V(x)\neq 0$.  

Now assume $|V|=n > 2$ and let $u \in V$ and write $V' = V- u$.
It suffices to show part (ii) since it implies that $|F_V(x)/F_{V-u}(x)|=|1+R(x)|\geq 1-a>0$, implying that $F_V(x)\neq 0.$
For notational convenience, we will write $F_U$ instead of $F_U(x)$ for $U\subseteq V'$ in what follows.

We establish a fact, namely \eqref{eq:telescopub} below, that we will use repeatedly. For $A = \{a_1, \ldots, a_k\}  \subseteq V'$, write $A_i = \{a_1, \ldots, a_i \}$ and $A_0 = \emptyset$. Then by induction we know that for each $i$, 
\[
\left| \frac{F_{V'-A_{i}}}{F_{V'-A_{i+1}}} - 1 \right| 
= \left| \frac{F_{V'-A_{i}}}{F_{V'-A_{i} - a_{i+1}}} - 1 \right|
< a 
\:\:\:\text{i.e.}\:\:\:
\frac{1}{1+a} \leq \left| \frac{F_{V'-A_{i+1}}}{F_{V'-A_{i}}} \right| \leq \frac{1}{1-a}.
\]
Using the upper bound and the fact that $F_{V'-A_i}\neq 0$ for each $i$ by induction, we have
\begin{equation}\label{eq:telescopub}
\left| \frac{F_{V'-A}}{F_{V'}} \right| 
= \prod_{i=0}^{k-1} \left| \frac{F_{V'-A_{i+1}}}{F_{V'-A_{i}}} \right| 
\leq (1-a)^{-k} = (1-a)^{-|A|}.
\end{equation}

Returning to the induction step for the claim, fix an ordering on the edges of $G$ such that the edges incident with $u$ are largest in the ordering. In particular, let $s_1, \ldots, s_{\ell}$ be the neighbours of $u$ (so $\ell \leq \Delta$) and assume that $us_1 > us_2 > \cdots > us_{\ell}$ are largest in the ordering and that the remaining edges are ordered arbitrarily.

Now, starting with \eqref{eq:R}, we can express $R$ as follows:
\begin{equation}
\label{eq:R2}
R  = \sum_{\substack{T \in \mathcal{F}_{V,u} \\ |T| \geq 1 }}x^{|T|}\frac{F_{V'-V(T)}}{F_{V'}}
= \sum_{\substack{S \subseteq N(u) \\ S \not= \emptyset}} x^{|S|}
\sum_{\substack{F \in \mathcal{F}^*_{V',S} \\ F \cup uS \text{ is BCF} }}
x^{|F|} \frac{F_{V' - S - V(F)}}{F_{V'}}\, .
\end{equation}
Here we write $F \cup uS$ to mean  $F \cup \{us: s \in S\}$. The second equality holds because $\mathcal{F}_{V,u} = \{F \cup uS: S \subseteq N_G(u), F \in \mathcal{F}^*_{V',S} \}$ (by our choice of edge ordering).
We have also used that if $F\cup uS = T \in \mathcal{F}_{V,u}$ then $V' -V(T) = V' - S - V(F)$.

Next we compare the inner sum in \eqref{eq:R2} with the same sum but over $F \in \mathcal{F}_{V', S}$ (which is equal to $1$ by \eqref{eq:recurF}). So we have
\begin{align}
    \Bigg|1 -  \sum_{\substack{F \in \mathcal{F}^*_{V',S} \\ F \cup uS \text{ is BCF}  }}
x^{|F|} \frac{F_{V' - S - V(F)}}{F_{V'}} \Bigg| 
&=
 \Bigg|\sum_{\substack{F \in \mathcal{F}_{V',S} }}
x^{|F|} \frac{F_{V' - S - V(F)}}{F_{V'}}  -  \sum_{\substack{F \in \mathcal{F}^*_{V',S} \\ F \cup uS \text{ is BCF}  }}
x^{|F|} \frac{F_{V' - S - V(F)}}{F_{V'}} \Bigg| \notag \\
&= 
 \Bigg| \sum_{F \in \mathcal{A} \cup \mathcal{B}}
x^{|F|} \frac{F_{V' - S - V(F)}}{F_{V'}} \Bigg|\, , \label{eq:ABbound}
\end{align}
where $\mathcal{A} = \mathcal{F}_{V',S} \setminus \mathcal{F}^*_{V',S}$, i.e.\ the set of $F \in \mathcal{F}_{V', S}$ where some component of $F$ contains (at least) two vertices of $S$ and $\mathcal{B}$ is the set of $F \in \mathcal{F}^*_{V', S}$ such that $F\cup uS$ contains a broken circuit, i.e.\ where some component of $F$ contains some $s \in S$ and some $t \in N(u)$ with $us < ut$ (this is the only way a broken circuit can occur by Proposition~\ref{pr:BFStree}). Now continuing from the expression above and applying the upper bound~\eqref{eq:telescopub} we have
\begin{align*}
    \left| \sum_{F \in \mathcal{A}\cup\mathcal{B}}
x^{|F|} \frac{F_{V' - S - V(F)}}{F_{V'}} \right| 
\leq  \sum_{F \in \mathcal{A}\cup\mathcal{B}}
|x|^{|F|} \left| \frac{F_{V' - S - V(F)}}{F_{V'}} \right|
&\leq \sum_{F \in \mathcal{A}\cup\mathcal{B}}
|x|^{|F|} (1-a)^{-|V(F) \cup S|} \\
&\leq \sum_{F \in \mathcal{A}\cup\mathcal{B}}
\left( \frac{|x|}{1-a}\right)^{|F|}  (1-a)^{-|S|},
\end{align*}
where we have used that $|V(F) \cup S| \leq |F| + |S|$ for all $F \in \mathcal{A}\cup\mathcal{B}$ since $(V(F) \cup S, F)$ has at most $|S|$ components. 
For $s,t\in N(u)$ we write $s<t$ if $us<ut$ for brevity.
Note that every $F \in \mathcal{A} \cup \mathcal{B}$ can be obtained by taking two vertices $s \in S$ and $t \in N(u)$ with $t>s$, and taking $F$ to be the union of some $T \in \mathcal{F}_{V',s,t}$ (defined in Section~\ref{sec:BCFPrelim}) with some $F' \in \mathcal{F}_{V', S-s}$ (although this procedure can also give objects that are not in $\mathcal{A}\cup\mathcal{B}$).
Writing $y := |x|/(1-a) \leq \ln b/(\Delta b)$, we can thus bound the right hand side of the expression above as
\begin{align}
\sum_{F \in \mathcal{A}\cup\mathcal{B}}
\left( \frac{|x|}{1-a}\right)^{|F|}  (1-a)^{-|S|}
&\leq \sum_{\substack{(s,t)\in S\times N(u)  \\ t>s}}  \sum_{T \in \mathcal{F}_{V',s,t}}  \sum_{F' \in \mathcal{F}_{V',S-s}} y^{|T| + |F'|} (1-a)^{- |S|} \nonumber 
\\ &= (1-a)^{-|S|} 
 \sum_{\substack{(s,t) \in S\times N(u)  \\  t>s}} 
 \sum_{T \in \mathcal{F}_{V',s,t}} y^{|T|}  \sum_{F' \in \mathcal{F}_{V',S-s}} y^{|F'|} \nonumber \\
 &\leq  (1-a)^{-|S|}\sum_{\substack{(s,t)\in S\times N(u)  \\  t>s}}  T_{G-u, s, t}(y) \prod_{t \in S-s}T_{G-u,t}(y) \nonumber \\
 &\leq 
 \sum_{\substack{(s,t)\in S\times N(u) \\ t>s}} 
 \left( \frac{b}{1-a} \right)^{|S|} \frac{ (\ln b)^{g-3}D(b)}{b}, \label{eq:bound on sum over AcupB}
\end{align}
assuming $G$ has girth at least $g$, where for the penultimate inequality we have used \eqref{eq:Forest-to-tree}.
For the last inequality, we have used Lemmas~\ref{lem:tree} and~\ref{lem:DR-tree} as well as the fact that there is no path in $G$ between vertices in $S$ with $g-2$ or fewer vertices since such a path together with $u$ would create a cycle of length at most $g-1$.

Returning to \eqref{eq:ABbound}, we deduce that if $G$ has girth at least $g$,
\begin{align}\label{eq:bound on *BCFforest sum}
 \Bigg| \sum_{\substack{F \in \mathcal{F}^*_{V',S} \\ F \cup uS \text{ is } BCF }}
x^{|F|} \frac{F_{V' - S - V(F)}}{F_{V'}}  \Bigg| 
\leq 1 + \sum_{\substack{(s,t)\in S\times N(u) \\  t>s}} \left( \frac{b}{1-a} \right)^{|S|} \frac{(\ln b)^{g-3}D(b)}{b},
\end{align}
and so using~\eqref{eq:R2},~\eqref{eq:bound on *BCFforest sum}, 
\begin{align*}
|R|  
&\leq \sum_{\substack{S \subseteq N(u) \\ S \not= \emptyset}} |x|^{|S|}
 \Bigg| \sum_{\substack{F \in \mathcal{F}^*_{V',S} \\ F \cup uS \text{ is } BCF }}
x^{|F|} \frac{F_{V' - S - V(F)}}{F_{V'}}  \Bigg| \\
&\leq \sum_{\substack{S \subseteq N(u) \\ S \not= \emptyset}} |x|^{|S|} +
\sum_{\substack{S \subseteq N(u) \\ S \not= \emptyset}}
\sum_{\substack{(s,t)\in S \times N(u) \\ t>s}} \left( \frac{|x|\cdot b}{1-a} \right)^{|S|}\frac{(\ln b)^{g-3}D(b)}{b},
\end{align*}
which can be further bounded using Lemma~\ref{lem:derivcount}, and $|x| \leq h(a,b) / \Delta = (1-a) (\ln b) /(\Delta b)$, by
\begin{align*}
&(1+|x|)^{\Delta} - 1 + \frac{(\ln b)^{g-3}D(b)}{b} \left[\binom{\Delta}{2} \left( \frac{\ln b}{\Delta} \right) \left(1 + \frac{\ln b}{\Delta} \right)^{\Delta - 1} - 1 \right] \\
<&\exp\left(\frac{(1-a)\ln b}{b}\right) - 1 + \frac{b(\ln b)^{g-1}}{2(1 - \ln b)} \\
=& f_g(a,b) \leq a,
\end{align*}
proving the claim and the theorem.
\end{proof}

\begin{remark}
Our proof strategy can also be used to give a compact and simple proof of the result of F\'ernandez and Procacci~\cite{Fernandez}. Indeed, just set $b=1+a$ and follow the start of our proof and as one arrives at~\eqref{eq:R2},  simply bound the inner sum by $((1+a)/(1-a))^{|S|}$ using~\eqref{eq:Forest-to-tree} and Lemma~\ref{lem:tree}. The bound then follows from minimizing $ \tfrac{1+a}{(1-a)\ln(a)}$ for $a<1$ and yields the constant less of~\cite{Fernandez}, cf.~\cite[Section 4.3.3]{PRsurvey}.
\end{remark}

We end the section with a proof of Theorem~\ref{thm:large girth}.

\begin{proof}[Proof of Theorem~\ref{thm:large girth}]
For $g\geq 3$, let 
\[
X_g =\{(a,b): f_g(a,b) \leq a,\, a\in [0,0.9],\, b\in [1.1,e)\}\, ,
\]
and let
\[
K_g = \min_{(a,b)\in X_g} \frac{1}{h(a,b)} = \min_{(a,b)\in X_g} \frac{b}{(1-a)\ln b}\, .
\]
First note that $X_g$ is non-empty (e.g. by the first line of Figure 1 and the fact that $X_3\subseteq X_g$).
Note also that the constraint $f_g(a,b)\leq a$ implies that $b\leq e-\epsilon_g$ for some $\epsilon_g>0$. It follows that $K_g$ is well-defined as the minimum of a continuous function over the non-empty compact set $X_g$.  

Since $X_3\subseteq X_4\subseteq \ldots$, we see that $(K_g)$ is a monotone decreasing sequence. Since $K_g\geq 0$ for all $g\geq 3$ we conclude that $\lim_{g\to\infty}K_g$ exists.

Now let 
\[
X_\infty=\left\{(a,b): \exp\left(\frac{(1-a)\ln b}{b}\right) - 1\leq a, a\in[0,0.9], b\in [1.1,e]\right\}\, ,
\]
and let $K_\infty=\min_{(a,b)\in X_\infty}1/h(a,b)$ (again this is well-defined since $X_\infty$ is non-empty and compact). We  first find an explicit expression for $K_\infty$ and then show that  $K_\infty=\lim_{g\to\infty}K_g$. The result will then follow from Theorem~\ref{thm:fgbd}.

 Let $(a^*, b^*)\in X_\infty$ be such that $K_\infty=1/h(a^*, b^*)$. Since $b^\ast > 1$, we must have $a^\ast>0$ by the first constraint in the definition of $X_\infty$. It follows that 
\[
 \exp\left(\frac{(1-a^*)\ln b^*}{b^*}\right) - 1= a^*
\]
else we could decrease the value of $a^\ast$ slightly and obtain a smaller value of the objective function $1/h$. Rearranging, we obtain
\begin{align}\label{eq:constraintrearrange}
    \frac{\ln b^*}{b^*}= \frac{\ln(1+a^*)}{1-a^*}\, ,
\end{align}
so that $h(a^*, b^*)=\ln(1+a^*)$. Our task is therefore to choose $a^*\in[0,0.9]$ maximal whilst satisfying the constraint~\eqref{eq:constraintrearrange} for some $b^\ast\in [1.1,e]$.
Since $\ln(x)/x$ is increasing on the interval $[1.1,e]$ and $\ln(1+x)/(1-x)$ is increasing on $[0,0.9]$, we conclude from~\eqref{eq:constraintrearrange} that $b^\ast=e$ and $a^*$ is the unique solution to 
\[
\frac{1}{e}= \frac{\ln(1+x)}{1-x}
\]
i.e. 
\[
a^*= e W(e^{2/e-1})-1 \approx 0.295741\, ,
\]
where $W$ denotes the the Lambert $W$ function\footnote{For $x>0$, $W(x)$ is the unique solution to the equation $We^W=x$.}. Since $h(a^*, b^*)=\ln(1+a^*)$ we conclude that
\[
K_\infty= \frac{1}{h(a^*, b^*)} = \frac{1}{1+\ln\left(W(e^{2/e-1})\right)}\approx 3.85977\, .
\]

It remains to show that $K_\infty=\lim_{g\to\infty}K_g$.

First note that for $g\geq 3$, $X_g\subseteq X_\infty$ and so $K_g\geq K_\infty$. In particular, $\lim_{g\to\infty}K_g\geq K_\infty$.
On the other hand, for any $\epsilon\in(0,1)$, we have $(a^\ast, b^\ast-\epsilon)\in X_g$ for $g$ sufficiently large so that
\[
\lim_{g\to\infty} K_g\leq 1/h(a^\ast, b^\ast-\epsilon)\, .
\]
Since $h$ is continuous, taking $\epsilon\to 0$, we see that $\lim_{g\to\infty}K_g\leq 1/h(a^\ast, b^\ast)=K_\infty$.


\end{proof}

\section{The forest generating function}\label{sec:forests}
Implicit in our proof of Theorem~\ref{thm:fgbd} above is a bound for the zeros of the forest generating function (defined at~\eqref{eq:forestgnfn}). Explicitly, in~\eqref{eq:ABbound} one can restrict to just the set $\mathcal{A}$.
With this approach we cannot obtain better bounds than the bounds we obtained for the BCF forest generating function.
However, with a different argument we can prove the much stronger bound of Theorem~\ref{thm:forest}.

Before we proceed we require a lemma about a path generating function.
We emphasize that in this section we have to allow multiple edges and loops and thus work with multigraphs.
For a multigraph $G$ and two of its vertices $u,v$ we denote by $\mathcal{P}_{G,u,v}$ the collection of all (edge sets of) paths in $G$ from $u$ to $v$.

\begin{lemma}\label{lem:path generating}
Let $G=(V,E)$ be a multigraph and let $v\in V$. Assume that all vertices but $v$ have degree at most $\Delta\in \mathbb{N}$.
Then for any vertex $u\neq v$ and any $c\in [0,1]$,
\[
\sum_{P\in \mathcal{P}_{G,u,v}} \left(\tfrac{c}{\Delta}\right)^{|P|}\leq c.
\]
\end{lemma}
\begin{proof}
The proof is by induction on the number of vertices of $G$. If this number is two, then the statement is clearly true.
Next suppose that $G$ has three or more vertices and let $u\in V$, $u\neq v$.
Let $e_1,\ldots,e_d$ be the edges between $u$ and $v$ (if there is no edge between $u,v$ we set $d=0$) and let $u_{1},\ldots,u_{D}$ be the remaining neighbours of $u$, counted with multiplicity. Note that $d+D\leq \Delta$.
Write $x=\tfrac{c}{\Delta}$.
Then, by induction,
\begin{align*}
\sum_{P\in \mathcal{P}_{G,u,v}}x^{|P|}&=\sum_{i=1}^d x+\sum_{i=1}^{D}x\sum_{P\in \mathcal{P}_{G-u,u_i,v}}x^{|P|}
\\
&\leq dx+xDc= d\tfrac{c}{\Delta}+D\tfrac{c^2}{\Delta}\leq c,
\end{align*}
as desired.
\end{proof}
We note that the lemma can be improved if $G$ is assumed to be a simple graph, but in our proof of Theorem~\ref{thm:forest} below we need it in its present form.
\begin{proof}[Proof of Theorem~\ref{thm:forest}]
We start by proving an identity. 
For an edge $e=uv$ of $G$ it holds that
\begin{equation}\label{eq:identity forest gen}
Z_G(x)=Z_{G/e}(x)+\sum_{P\in \mathcal{P}_{G,u,v}}x^{|P|}Z_{G/P}(x), 
\end{equation}
where $G/e$ and $G/P$ denote the multigraphs obtained from contracting the edge $e$, respectively the path $P$ (meaning that we keep any parallel edges/loops that arise).
To prove this, note that a forest contributing to $Z_G$ either contains $u$ and $v$ in the same component or not.
The latter forests correspond bijectively to the forests of $G/e$ via $F\mapsto F'=(F+e)/e$ and satisfy $|F|=|F'|$.
The forests $F$ for which $u$ and $v$ are in the same component can be uniquely decomposed into a path $P$ from $u$ to $v$ and a forest $F''$ in $G/P$ such that $|F|=|P|+|F''|$ and conversely a path $P$ from $u$ to $v$ together with a forest $F''$ in $G/P$ yields a forest in $G$ for which $u$ and $v$ are in the same component.
This shows the desired decomposition and proves~\eqref{eq:identity forest gen}.

Henceforth we assume $|x|\leq 1/(2\Delta)$. Using~\eqref{eq:identity forest gen} we now prove the following claim by induction on the number of edges, which directly implies the theorem. 
\begin{itemize}
    \item[(i)] For any edge $e=uv$, where $v$ is the vertex of potentially unbounded degree, we have
    \[\left |\frac{Z_{G}(x)}{Z_{G/e}(x)}-1\right|\leq 1/2,\]
    \item[(ii)] $Z_{G}(x)\neq 0$.
\end{itemize}
If $G$ has no edges then $Z_{G}(x)=1$ and (i) is vacuous.
Next assume that $G$ has at least one edge and let $v$ be the vertex of largest degree and let $uv\in E$.
Clearly (ii) follows from (i) and so it suffices to show (i).
Since by induction $Z_{G/e}(x)\neq 0$ the ratio $\frac{Z_{G}(x)}{Z_{G/e}(x)}$ is well defined and by~\eqref{eq:identity forest gen} it satisfies,
\begin{align}\label{eq:step 1 forest}
\left|\frac{Z_{G}(x)}{Z_{G/e}(x)}-1\right|\leq \sum_{P\in \mathcal{P}_{G,u,v}} |x|^{|P|}\left|\frac{Z_{G/P}(x)}{Z_{G/e}(x)}\right|.    
\end{align}
For a path $P$ contributing to the sum above let $f=vw$ be its last edge and let $P^*$ be the path obtained from $P$ by removing $f$ from $P$ and adding $e$. Then, since $G/P = G/P^*$ we have $Z_{G/P}(x)=Z_{G/P^*}(x)$. Next let $e_1=e$ and write $P^*=(e_1,e_2,\ldots,e_k)$ and let $P^*_i=(e_1,e_2,\ldots,e_i)$. We note that for each $i$, the multigraph $G/P^*_i$ has at most one vertex of unbounded degree, the vertex corresponding to $v$ after the contraction, and $G/P^*_{i}=(G/P^*_{i-1})/e_i$. 
Therefore by induction and (ii) we can write 
\[
\frac{Z_{G/P^*}(x)}{Z_{G/e}(x)}=\prod_{i=1}^{k-1}\frac{Z_{G/P^*_{i+1}}(x)}{Z_{G/P^*_i}(x)}.
\]
Next by induction and (i), and the triangle inequality we have $\left|\frac{Z_{G/P^*_{i}}(x)}{Z_{G/P^*_{i+1}}(x)}\right|\geq 1/2$  and therefore $\left|\frac{Z_{G/P^*}(x)}{Z_{G/e}(x)}\right|\leq 2^{|P|-1}$.
Now returning to~\eqref{eq:step 1 forest}, we obtain, since $|x|\leq \tfrac{1}{2\Delta}$,
\[
\left|\frac{Z_{G}(x)}{Z_{G/e}(x)}-1\right|\leq 1/2\sum_{P\in \mathcal{P}_{G,u,v}} (2|x|)^{|P|}\leq 1/2,
\]
by Lemma~\ref{lem:path generating} and this finishes the proof.
\end{proof}

We note that our proof shows that Theorem~\ref{thm:forest} is in fact true for multigraphs.
In the introduction we mentioned that the radius of the zero-free disc in Theorem~\ref{thm:forest} cannot be replaced by $\tfrac{1}{\Delta-2}$. 
The multigraph consisting of two vertices connected with $\Delta$ parallel edges has $-1/\Delta$ as a zero of its forest generating function and so for multigraphs we cannot replace the radius by $1/\Delta$.


\subsection*{Acknowledgement}
We thank Ferenc Bencs for correcting an incorrect reference to~\cite{BencCsikvari22} in an earlier version.
We also thank the referees for their careful reading and suggestions.  

\section{Concluding remarks}
To conclude, we collect some remarks on possible extensions and further directions of research.

\subsubsection*{Bipartite graphs}
We can obtain better bounds on the chromatic zeros for bipartite graphs than the bounds we obtained for graphs of girth at least $4$.
This is because Lemma~\ref{lem:DR-tree} can be improved when all paths between $v_1$ and $v_2$ are of even length. The resulting bound can then be plugged into~\eqref{eq:bound on sum over AcupB}.
This can be used to show that if $G$ is bipartite, then its chromatic zeros lie in a disc of radius $K\Delta(G)$ where $K\leq 5.05$.

\subsubsection*{The double-rooted tree generating function}
Sokal~\cite{sokal} showed that for a graph $G$ of maximum degree $\Delta$ and vertex $v\in V(G)$, $T_{G,v}(x)\leq T_{\mathbb T, r}(x)$ for all $x>0$ where $\mathbb T$ is the infinite $\Delta$-regular tree with root $r$. This in turn implies Lemma~\ref{lem:tree}. It is less clear what graph $G$ and pair of vertices $v_1, v_2\in V(G)$ maximises the double-rooted tree generating function $T_{G, v_1, v_2}(x)$ for a given $x>0$. Further understanding of this problem would likely lead to an improvement on Lemma~\ref{lem:DR-tree} and in turn an improvement to the constant $K$ in Theorem~\ref{thm:main}.

\subsubsection*{Real zeros}
Based on the polymer framework, Dong and Koh~\cite{Dongreal} managed to prove a bound of $5.664\Delta$ for \emph{real} chromatic zeros of graphs of maximum degree at most $\Delta$. It seems likely that with  our approach this bound can be improved, but this is not automatic. 
In particular one would need to devise improved bounds on the (double) rooted tree generating functions evaluated at negative numbers.

\subsubsection*{Multivariate generating functions}
The reader can check that for both the forest and BCF forest generating functions our zero-freeness results extend to the multivariate setting where we equip each edge with its own variable. That is, for a graph $G=(V,E)$ we let $\mathbf{x}=(x_e : e\in E)$ and define
\begin{align*}
    Z_G(\mathbf{x}) = \sum_{\substack{F\subseteq E: \\(V,F) \text{ is a forest}}} \prod_{e\in F}x_e
\end{align*}
and define $F_G(\mathbf{x})$ similarly. Then, for example, one can show that $Z_G(\mathbf{x})$ has no zeros in the complex polydisc $\{\mathbf{x} : |x_e|\leq 1/(2\Delta) \text{ for all } e\in E\}$.

It is also possible to equip each tree $T$ in $G$ with its own variable $x_T$ and consider the generating function
\[
\hat Z_G(\mathbf{x}) = \sum_{\substack{F\subseteq E: \\(V,F) \text{ is a forest}}} \prod_{T\in F}x_T\, ,
\]
where $\mathbf x$ now denotes the vector of all the $x_T$'s and  we write $T\in F$ to mean that $T$ is a (non-trivial) component of $F$. We can define $\hat F_G(\mathbf x)$ similarly. It is however not clear how to extend our proof for the forest generating function (nor for the BCF forest generating function) in this setting. 
If possible somehow, this would be a way to leverage the result on the forest generating function to obtain results for the chromatic polynomial (since we can recover $\hat F_G(\mathbf x)$ from $\hat Z_G(\mathbf{x})$ by simply setting variables $x_T$ corresponding to non-BCF trees $T$ to 0).

\subsubsection*{The anti-ferromagnetic Potts model}
Sokal~\cite{sokal} extended his proof of the boundedness of the chromatic zeros to the partition function of the anti-ferromagnetic Potts model. It is not clear whether this is also possible with our approach and it would be interesting to see if a better bound can be obtained via our approach.

\subsubsection*{Optimal zero-free discs}
Theorem~\ref{thm:main} naturally raises the question (already posed by Sokal~\cite{sokal})
of determining the least $K$ for which the theorem holds. Perhaps an easier question to start with is to determine the minimum value of $K_\infty=\lim_{g\to\infty}K_g$ that one could achieve in Theorem~\ref{thm:large girth}. A more difficult question is to determine the minimum value of $K_g$ for each $g$ and the extremal graph or sequence of graphs that witness this minimum value.

\bibliographystyle{numeric}
\bibliography{chromatic}
\end{document}